\documentclass[12pt]{amsart}
\usepackage{fullpage,url,amssymb,citesort}
\newcommand{\F}{{\mathbb F}}

\newcommand{\Z}{{\mathbb Z}}

\newcommand{\eps}{\varepsilon}

\newcommand{\frakp}{\mathfrak{p}}

\newtheorem{theorem}{Theorem}[section]
\newtheorem{lemma}[theorem]{Lemma}

\theoremstyle{definition}

\newtheorem{question}[theorem]{Question}

\theoremstyle{remark}
\newtheorem{remark}[theorem]{Remark}

\begin{document}

\title{Everywhere ramified towers of global function fields}
%\keywords{9999}
\author{Iwan Duursma}
\address{Department of Mathematics, University of Illinois,
	Urbana, IL 61801--2975, USA}
\email{duursma@math.uiuc.edu}
\urladdr{http://www.math.uiuc.edu/\~{}duursma}

\author{Bjorn Poonen}
\thanks{The second author was supported in part by NSF grant DMS-0301280
          and a Packard Fellowship.}
\address{Department of Mathematics, University of California,
	Berkeley, CA 94720--3840, USA}
\email{poonen@math.berkeley.edu}
\urladdr{http://math.berkeley.edu/\~{}poonen}

\author{Michael Zieve}
\address{
Center for Communications Research,
805 Bunn Drive,
Princeton, NJ, 08540,
USA
}
\email{zieve@idaccr.org}

\date{July 16, 2003}

\begin{abstract}
We construct a tower of function fields $F_0\subset F_1\subset\dots$ over
a finite field such that every place of every $F_i$ ramifies in the tower
and $\lim \text{genus}(F_i)/[F_i:F_0]<\infty$.
We also construct a tower in which every place ramifies and
$\lim N_{F_i}/[F_i:F_0]>0$, where $N_{F_i}$ is the number of degree-$1$
places of $F_i$.
These towers answer questions posed by Stichtenoth at Fq7.
\end{abstract}

\maketitle

%****************************************************************************
\section{Introduction}\label{S:introduction}

Let $q$ be a prime power, and let $\F_q$ be a finite field of size $q$.
By a \emph{function field over $\F_q$},
we mean a finitely generated extension $K/\F_q$
of transcendence degree~1 in which $\F_q$ is algebraically closed.
By an \emph{extension of function fields $K'/K$},
we mean a finite separable extension such that $K$ and $K'$
are function fields over the same $\F_q$.
Let $g_K$ be the genus of $K$.
Let $N_K$ be the number of degree-$1$ places of $K$
(the number of $\F_q$-rational points on the corresponding curve).
A \emph{tower} of function fields over $\F_q$ is a sequence
of extensions of such function fields
\begin{equation*}
K_0 \subset K_1 \subset K_2 \subset \dots
\end{equation*}
such that $g_i:=g_{K_i}\to\infty$ as $i\to\infty$.
Define $N_i=N_{K_i}$, and $d_i = [K_i:K_0]$.
Since $N_i/d_i$ is decreasing while $(g_i-1)/d_i$ is increasing (Hurwitz),
$\lim N_i/d_i$  and $\lim g_i/d_i$  exist.
(The latter can be $\infty$.)

The Weil bound $N_K\le q+1+2g_K\sqrt{q}$ implies
\begin{equation*}
\lim N_i/g_i \le 2\sqrt{q}.
\end{equation*}
This was improved by Drinfeld and Vladut \cite{DV} (following Ihara \cite{Ih})
to
\begin{equation*}
\lim N_i/g_i \le \sqrt{q}-1.
\end{equation*}
Ihara also showed that, for any square $q$, there are towers of Shimura curves
with $\lim N_i/g_i=\sqrt{q}-1$ \cite{Ih-66,Ih-69,Ih-75,Ih-79,Ih}.
Subsequent authors have given further constructions
of `asymptotically good' towers,
i.e., towers with
$\lim N_i/g_i>0$ \cite{AM,Be,BG,E,E-Drinfeld,FKV,FPS,GS1,GS2,GS4,%
GST,GV,HM,LM,LMS,MV,NX1,NX2,NX3,NX4,Pe,Sc,Se,Se-Harvard,St,Te,Wu,Xi,Zi}.

Every known asymptotically good tower has two special properties:
there is some place of some $K_i$ which splits completely in the tower,
and there are only finitely many places of $K_0$ which ramify in the tower.
(We say that a place of $K_i$ \emph{splits completely in the tower}
if it splits completely in $K_j/K_i$ for every $j\ge i$.
We say that a place of $K_0$ \emph{ramifies in the tower}
if there exists $i$ such that it ramifies in $K_i/K_0$.)
But it is difficult to study asymptotically good towers directly
since one must control both the genus and the number of rational places.
With this as motivation,
Stichtenoth posed the following two questions in his talk at Fq7
(the Seventh International Conference on Finite Fields and Their Applications):

\begin{question}\label{Q:splits completely}
If $\lim N_i/d_i > 0$, must some $K_i$
have a place that splits completely in the tower?
\end{question}

\begin{question}\label{Q:finite ramification}
If $\lim g_i/d_i <\infty$,
must only finitely many places of $K_0$ ramify in the tower?
\end{question}

Our Theorems \ref{T:ramified tower1} and~\ref{T:ramified tower2}
imply negative answers to these two questions.
Call a tower $K_0 \subset K_1 \subset \dots$
of function fields over $\F_q$
\emph{everywhere ramified}
if for each place $P$ of each $K_i$,
there exists $j>i$ such that $P$ ramifies in $K_j/K_i$.

\begin{theorem}
\label{T:ramified tower1}
Given a function field $K_0$ over $\F_q$ with a rational place,
there exists an everywhere ramified tower
$K_0 \subset K_1 \subset \dots$
such that $\lim N_i/d_i > 0$.
\end{theorem}

\begin{theorem}
\label{T:ramified tower2}
Given a function field $K_0$ over $\F_q$,
there exists an everywhere ramified tower
$K_0 \subset K_1 \subset \dots$
such that $\lim g_i/d_i <\infty$.
\end{theorem}

%****************************************************************************
\section{Proof of Theorem~\ref{T:ramified tower1}}\label{S:answer1}

\begin{lemma}\label{L:all split completely}
Let $K$ be a function field over $\F_q$.
Then there is a nontrivial extension $K'/K$
in which all rational places of $K$ split completely.
\end{lemma}

\begin{proof}
Weak approximation (or Riemann-Roch) gives $f \in K^*$
having a zero at each rational place of $K$ and a simple pole at some
other place of $K$.
Adjoin a root of $y^q-y=f$ to obtain $K'$.
Then $K'/K$ is totally ramified above the simple pole of $f$,
so $K'$ is another function field over $\F_q$ and $[K':K]=q > 1$.
\end{proof}

\begin{lemma}\label{L:most split completely}
Let $K$ be a function field over $\F_q$ with $N_K>0$,
and let $P$ be a place of $K$.
For any $\eps > 0$, there is an extension $L/K$ such that
$N_L/N_K > (1-\eps) [L:K]$
and $P$ ramifies in $L/K$.
\end{lemma}

\begin{proof}
We first reduce to the case where $1/N_K < \eps$.
Repeated application of Lemma~\ref{L:all split completely}
yields $K'/K$ such that $1/([K':K] N_K)< \eps$ and all
rational places of $K$ split completely.
Then $N_{K'} = [K':K] N_K$.
Pick a place $P'$ of $K'$ above $P$.
If we could find $L/K'$ satisfying the conditions
of the lemma for $(K',P')$,
then
\[
	\frac{N_L}{N_K} = \frac{N_L}{N_{K'}} \frac{N_{K'}}{N_K}
		> (1-\eps)[L:K'] [K':K] = (1-\eps)[L:K],
\]
so $L/K$ would work for $(K,P)$.
Thus, renaming $K'$ as $K$, we may assume $1/N_K < \eps$.

Weak approximation gives $f \in K^*$ having a simple pole at $P$
and zeros at all rational places not equal to $P$.
Adjoin a root of $y^q-y=f$ to obtain $L$.
Then $P$ ramifies in $L/K$, but all other rational places of $K$
split completely, so $N_L \ge (N_K-1)q$.
Thus $N_L/N_K \ge q (1-1/N_K) > [L:K](1-\eps)$.
\end{proof}

\begin{proof}[Proof of Theorem~\ref{T:ramified tower1}]
Fix a sequence of positive numbers $\eps_m \to 0$
such that $\prod_{m=1}^\infty (1-\eps_m)$ converges to a positive number.
In our proof we will apply Lemma~\ref{L:most split completely}
infinitely often, using $\eps_1$ in the first application,
$\eps_2$ in the second application, and so on.

Let $P_0, P_1, \dots$ be an enumeration of the places of $K_0$
(of all degrees).
Given $K_i$, we construct $K_{i+1}$ in stages so that all places of $K_i$
lying above $P_0,\dots,P_i$ ramify in $K_{i+1}/K_i$.
Namely, if $Q_1,\dots,Q_I$ are all the places of $K_i$
lying above $P_0,\dots,P_i$,
we set $K_{i,0}=K_i$
and then for $j=1,\dots,I$ in turn,
apply Lemma~\ref{L:most split completely} with the first
unused $\eps_m$ to find $K_{i,j}/K_{i,j-1}$
in which some place of $K_{i,j-1}$ above $Q_j$ ramifies
and $N_{K_{i,j}}/N_{K_{i,j-1}}>(1-\eps_m)[K_{i,j}:K_{i,j-1}]$.
Finally, set $K_{i+1}=K_{i,I}$.

If $R$ is a place of some $K_r$,
then $R$ lies over some $P_j$ of $K_0$.
By construction, for all $i \ge \max\{j,r\}$,
all places of $K_i$ above $R$ ramify in $K_{i+1}/K_i$.
Thus $R$ is ramified in $K_{i+1}/K_r$.

The inequality in Lemma~\ref{L:most split completely}
guarantees that the value of $N/d$ for $K_{i,j}$
is at least $1-\eps_m$ times the value of $N/d$ for $K_{i,j-1}$.
Thus $N_i/d_i$ is at least $\left(\prod_{m \le M} (1-\eps_m) \right) N_0/d_0$,
if $M$ is the number of applications of Lemma~\ref{L:most split completely}
used in the construction up to $K_i$.
Since $N_0/d_0>0$ and $\prod_{m=1}^\infty (1-\eps_m)$ converges,
the decreasing sequence $N_i/d_i$ is bounded below by
\[
	\left(\prod_{m=1}^{\infty} (1-\eps_m) \right) N_0/d_0,
\]
which is positive.
So $N_i/d_i$ has a positive limit.
Finally, $N_i \to \infty$ implies $g_i \to \infty$.
\end{proof}

\begin{remark}
A slight modification of the argument shows that, given $K_0$,
we can construct an everywhere ramified tower in which
$N_i/d_i$ converges to any prescribed value less than $N_0$.
This is because weak approximation
lets us prescribe the ramification and splitting
of any finite number of places at each step.
\end{remark}

%****************************************************************************
\section{Proof of Theorem~\ref{T:ramified tower2}}\label{S:answer2}

Let $p$ be the characteristic of $\F_q$.

\begin{lemma}\label{L:everywhere unramified}
Let $K$ be a function field over $\F_q$ of genus $>1$,
and let $P$ be a place of $K$.
Then there exist unramified extensions $K'/K$
of arbitrarily high genus
such that for some place $Q$ of $K'$ lying over $P$,
the residue field extension for $Q/P$ is trivial.
\end{lemma}

\begin{proof}
Let $C$ be the smooth, projective, geometrically integral curve
with function field $K$.
Let $J$ be the Jacobian of $C$.
There exists a degree-1 divisor $D$ on $C$ \cite[V.1.11]{Stb}.
Use $D$ to identify $C$ with a closed subvariety of $J$.

The place $P$ corresponds to a Galois conjugacy class of
points in $C(\F_{q^f})$,
where $\F_{q^f}$ is the residue field.
Choose $P_0$ in this conjugacy class.
Choose $n \in \Z_{>0}$ such that $n \equiv 1 \pmod{p \cdot \#J(\F_{q^f})}$.
Then the multiplication-by-$n$ map $[n]\colon J \to J$
is \'etale, and maps $P_0$ to itself.
Let $C' = [n]^{-1} C$, so $C'$ is an \'etale cover of $C$.
Then $C'$ corresponds to a function field $K'$
that is unramified over $K$.
Also $P_0 \in C'(\F_{q^f})$
represents a place $Q$ of $K'$ lying over $P$,
having the same residue field as $P$.
By choosing $n$ large,
we can make $g_{K'}$ as large as desired,
by the Hurwitz formula.
\end{proof}

\begin{lemma}\label{L:slightly ramified}
Let $K$ be a function field over $\F_q$ of genus $>1$,
let $P$ be a place of $K$,
and let $\eps>0$.
Then there exists an extension $L/K$ with
$(g_L-1)/(g_K-1) < (1+\eps)[L:K]$
such that $P$ ramifies in $L/K$.
\end{lemma}

\begin{proof}
Let $f$ be the degree of $P$ over $\F_q$.
For an unramified extension $K'/K$,
we have $(g_{K'}-1)/(g_K-1) = [K':K]$ by Hurwitz.
By applying Lemma~\ref{L:everywhere unramified},
we may replace $(K,P)$ by some $(K',Q)$
in order to assume that $g_K$ is arbitrarily large,
without changing $f$.

When $g_K$ is sufficiently large,
an easy estimate (e.g.\ cf.\ \cite[V.2.10]{Stb})
based on the Weil bounds
implies there exist places $Q,Q'$ of $K$
of degrees $d,d+f$ respectively, where $d$ is
the smallest integer $>\sqrt{g_K}$ and not equal to $f$.
Choose a prime $\ell \nmid p \cdot \#G$, where $G$ is the group of degree-zero
divisor classes of $K$.  Then every element of $G$, and in particular $[Q'-Q-P]$,
is divisible by $\ell$.
Thus, there exists a divisor $D$ of degree~$0$ and an element $h$ of $K$
such that $(h) = Q'-Q-P - \ell D$.
Let $L=K(h^{1/\ell})$, so $[L:K]=\ell$.
Hurwitz gives
\[
	2g_L-2 = \ell(2g_K-2) + (\ell-1)((d+f)+d+f),
\]
so
\[
	\frac{g_L-1}{[L:K](g_K-1)}
	\quad=\quad 1 + \frac{\ell-1}{\ell} \left( \frac{d+f}{g_K-1} \right)
	\quad=\quad 1 + O(g_K^{-1/2}).
\]
The $O(g_K^{-1/2})$ term will be $<\eps$ if $g_K$ is sufficiently large.
\end{proof}

\begin{proof}[Proof of Theorem~\ref{T:ramified tower2}]
Given $K_0$, let $K_1/K_0$ be an extension with $g_1>1$.
Just as Lemma~\ref{L:most split completely}
let us prove Theorem~\ref{T:ramified tower1},
Lemma~\ref{L:slightly ramified}
now lets us construct an everywhere ramified tower
$K_1 \subset K_2 \subset \dots$
such that at the $i^{{\operatorname{th}}}$ step
the value of $(g_i-1)/d_i$ increases by a factor at most $1+\eps_i$
for a prescribed $\eps_i>0$.
By choosing $\eps_i$ so that $\prod (1+\eps_i)$ converges,
we obtain such a tower with
$\lim (g_i-1)/d_i < \infty$.
Since $d_i \to \infty$, this limit equals $\lim g_i/d_i$.
\end{proof}

%****************************************************************************
\section{Question}\label{S:question}

Can one combine Theorems \ref{T:ramified tower1} and~\ref{T:ramified tower2}?
In particular, does there exist an everywhere ramified tower
in which both $\lim N_i/d_i > 0$ and $\lim g_i/d_i <\infty$?

\end{document}